\documentclass[12pt]{article}
\usepackage[margin=2.5cm]{geometry}
\usepackage{array}
\usepackage[fleqn]{amsmath}
\usepackage{amssymb,latexsym}
\usepackage{amsthm}
\usepackage[section]{placeins}
\usepackage[colorlinks=true,linkcolor=black,citecolor=black]{hyperref}
\usepackage{enumitem}
\usepackage{graphicx}
\usepackage{authblk}
\usepackage{tikz}

\newtheorem{theorem}{Theorem}[section]

\newtheorem{lemma}[theorem]{Lemma}

\newcommand{\Z}{\mathbb{Z}}

\DeclareMathOperator{\Diff}{Diff}
\title{Graphs derived from perfect difference sets}

\author[1]{Grahame Erskine}
\author[2]{Peter Fratri\v{c}}
\author[1,2]{Jozef \v{S}ir\'a\v{n}}

\affil[1]{Open University, Milton Keynes, UK}
\affil[2]{Slovak University of Technology, Bratislava, Slovakia}
\date{}

\begin{document}

\maketitle
\let\thefootnote\relax
\let\thefootnote\relax

\begin{abstract}
\noindent We study a family of graphs with diameter two and asymptotically optimal order for their maximum degree, obtained from perfect difference sets. We show that for all known examples of perfect difference sets, the graph we obtain is isomorphic to one of the Brown graphs, a well-known family of graphs in the degree-diameter problem.
\end{abstract}

\section{Introduction}\label{sec:intro}
The degree-diameter problem seeks to find the largest possible graph of diameter $k$ and maximum degree $\Delta$. In the case of diameter 2, a simple counting argument yields an upper bound of $\Delta^2+1$ for the number of vertices in a graph. Graphs attaining this bound (the \emph{Moore bound}) are necessarily regular and are known to be exceedingly rare; the only examples being the cycle $C_5$ with degree 2, the Petersen graph of degree 3 and the Hoffman-Singleton graph of degree 7. By a classical result in algebraic graph theory~\cite{Hoffman1960}, the only other possible graph would have degree 57, and this existence or otherwise of such a graph is a famous open problem. For much more on Moore graphs and the degree-diameter problem, the reader is referred to the survey~\cite{miller2005moore}.

Given the scarcity of Moore graphs, it is natural to consider instead families of graphs which are in some sense close to the Moore bound. In Section~\ref{sec:diffgraphs}, we describe a family of graphs of diameter two which asymptotically approach the Moore bound for certain values of the maximum degree $\Delta$. For many values of $\Delta$, the current largest known graphs of diameter two in the literature are the \emph{Brown graphs} (also known as \emph{polarity graphs}). In Section~\ref{sec:browngraphs} we describe the Brown graphs, and show that in fact our graphs are in all known cases isomorphic to one of the Brown graphs, even though their construction is quite different.

Our graphs are based on perfect difference sets, and before describing their construction we give some background on these interesting combinatorial objects. A \emph{perfect difference set} $S$ is a set of residues modulo $n$ (for some positive integer $n$) with the property that every non-zero residue modulo $n$ can be uniquely expressed as the difference of two elements of $S$. If $|S|=k$, it is immediate by counting pairs of elements of $S$ that $n=k^2-k+1$. If $S$ is a perfect difference set modulo $n$, then it is clear that $S+m$ (for any integer $m$) and $rS$ (for any positive integer $r$ with $\mathrm{gcd}(n,r)=1$) are also perfect difference sets. We call two perfect difference sets which are related in this way \emph{equivalent}.

In 1938, Singer~\cite{Singer1938} showed that a sufficient condition for a perfect difference set $S$ modulo $n$ to exist is that $n=q^2+q+1$ for some prime power $q$. The set $S$ then has size $q+1$. Singer's construction based on finite fields is crucial to form the link between our difference graphs and the Brown graphs in the degree-diameter problem, and we review the construction in Section~\ref{sec:diffgraphs}.

To date, no perfect difference set with $|S|$ not equal to one more than a prime power is known to exist. Such a set would lead immediately to a projective plane of non prime power order, the existence of which is one of the most famous open problems in combinatorics.

\section{Difference graphs}\label{sec:diffgraphs}
Let $S$ be a perfect difference set modulo $n$. We define the \emph{difference graph} $\Diff(\Z_n,S)$ as follows. The vertex set of $\Diff(\Z_n,S)$ is the set of residues modulo $n$, which we identify with the additive cyclic group $\Z_n$. For any $x,y\in\Z_n$, there is an edge from $x$ to $y$ in the graph if and only if $x+y\in S$. (We suppress the loops in the graph for any $x$ with $x+x\in S$.)

It is apparent from the definition that $\Diff(\Z_n,S)$ has order $n=q^2+q+1$, where $q+1=|S|$. A vertex $x$ has degree $q+1$, unless $x+x\in S$ in which case it has degree $q$. Thus $\Diff(\Z_n,S)$ has $q+1$ vertices of degree $q$ and $q^2$ vertices of degree $q+1$. If $x$ and $y$ are distinct vertices, then we may write $x-y=s-t$ for some $s,t\in S$. Then the vertex $s-x=t-y$ is adjacent to both $x$ and $y$. Thus $\Diff(\Z_n,S)$ has diameter 2, and since it has maximum degree $\Delta=q+1$ and order $q^2+q+1$, its order asymptotically approaches the Moore bound for large $q$.

The following lemma is easily proved.

\begin{lemma}\label{lem:iso}
Let $S$ and $T$ be equivalent perfect difference sets for the cyclic group $\Z_n$. Then the difference graphs $\Diff(\Z_n,S)$ and $\Diff(\Z_n,T)$ are isomorphic.
\end{lemma}

Figure~\ref{fig:diff4} shows an example of a difference graph of order 21, derived from the perfect difference set $S=\{0,1,4,14,16\}$ in $\Z_{21}$. We can see that the five vertices $0,2,7,8,11$ have degree 4, and the remainder have degree 5. In this case vertex 14 is adjacent to all the vertices of minimum degree, although this is not typical.
\begin{figure}[h]\centering
	\begin{tikzpicture}[x=0.2mm,y=-0.2mm,inner sep=0.1mm,scale=1,thick,vertex/.style={circle,draw,minimum size=12,fill=white}]
	\scriptsize
	\node at (460,200) [vertex] (v1) {$0$};
	\node at (600,500) [vertex] (v2) {$1$};
	\node at (220,200) [vertex] (v3) {$2$};
	\node at (600,360) [vertex] (v4) {$3$};
	\node at (160,360) [vertex] (v5) {$4$};
	\node at (260,380) [vertex] (v6) {$5$};
	\node at (320,580) [vertex] (v7) {$6$};
	\node at (380,200) [vertex] (v8) {$7$};
	\node at (540,200) [vertex] (v9) {$8$};
	\node at (420,320) [vertex] (v10) {$9$};
	\node at (160,500) [vertex] (v11) {$10$};
	\node at (300,200) [vertex] (v12) {$11$};
	\node at (80,280) [vertex] (v13) {$12$};
	\node at (680,280) [vertex] (v14) {$13$};
	\node at (380,120) [vertex] (v15) {$14$};
	\node at (380,660) [vertex] (v16) {$15$};
	\node at (500,380) [vertex] (v17) {$16$};
	\node at (540,500) [vertex] (v18) {$17$};
	\node at (340,320) [vertex] (v19) {$18$};
	\node at (220,500) [vertex] (v20) {$19$};
	\node at (440,580) [vertex] (v21) {$20$};
	\path
	(v1) edge (v2)
	(v1) edge (v5)
	(v1) edge (v15)
	(v1) edge (v17)
	(v2) edge (v4)
	(v2) edge (v14)
	(v2) edge (v16)
	(v2) edge (v21)
	(v3) edge (v13)
	(v3) edge (v15)
	(v3) edge (v20)
	(v3) edge (v21)
	(v4) edge (v12)
	(v4) edge (v14)
	(v4) edge (v19)
	(v4) edge (v20)
	(v5) edge (v11)
	(v5) edge (v13)
	(v5) edge (v18)
	(v5) edge (v19)
	(v6) edge (v10)
	(v6) edge (v12)
	(v6) edge (v17)
	(v6) edge (v18)
	(v6) edge (v21)
	(v7) edge (v9)
	(v7) edge (v11)
	(v7) edge (v16)
	(v7) edge (v17)
	(v7) edge (v20)
	(v8) edge (v10)
	(v8) edge (v15)
	(v8) edge (v16)
	(v8) edge (v19)
	(v9) edge (v14)
	(v9) edge (v15)
	(v9) edge (v18)
	(v10) edge (v13)
	(v10) edge (v14)
	(v10) edge (v17)
	(v11) edge (v12)
	(v11) edge (v13)
	(v11) edge (v16)
	(v12) edge (v15)
	(v13) edge (v14)
	(v16) edge (v21)
	(v17) edge (v20)
	(v18) edge (v19)
	(v18) edge (v21)
	(v19) edge (v20)
	;
	\end{tikzpicture}
	\caption{The difference graph $\Diff(\Z_{21},\{0,1,4,14,16\})$}
	\label{fig:diff4}
\end{figure}
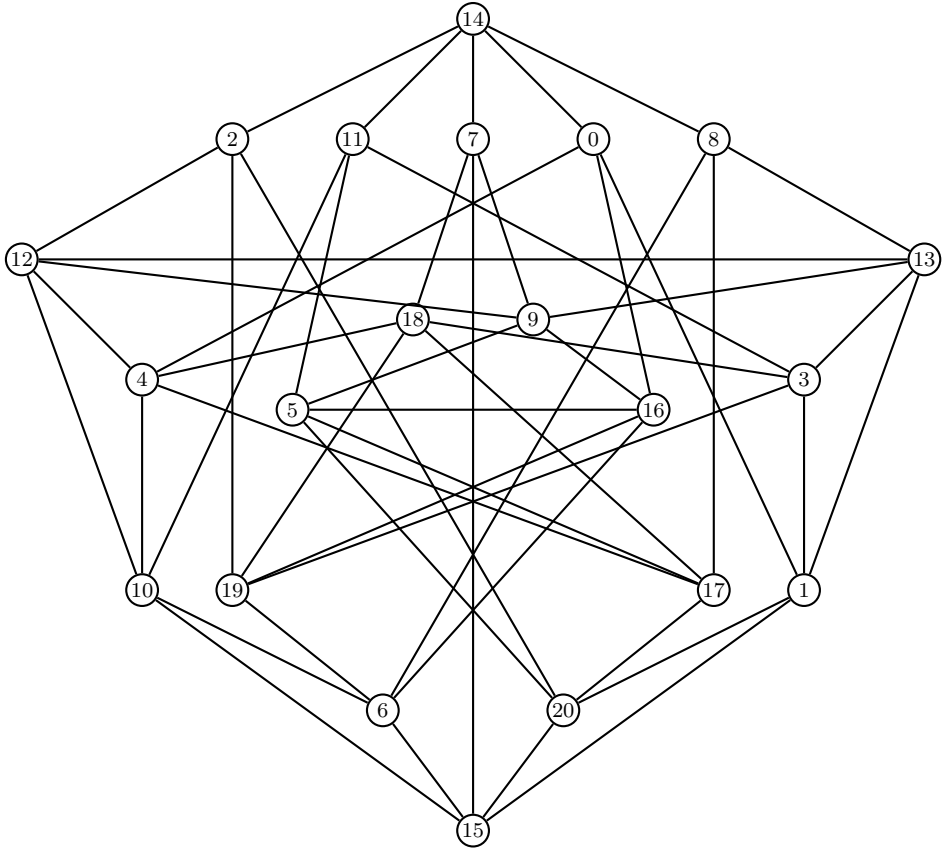

We now recast the definition of our difference graphs in terms of Singer's construction~\cite{Singer1938} of perfect difference sets, as amplified by Halberstam and Laxton~\cite{Halberstam1963} and others. Let $q$ be a prime power, and let $K=GF(q)$ be the unique finite field with $q$ elements. Let $F=GF(q^3)$, so that $K$ is a subfield of $F$. The multiplicative groups $K^*$ and $F^*$ are cyclic, of orders $q-1$ and $q^3-1$ respectively, and so the quotient group $G=F^*/K^*$ is cyclic of order $q^2+q+1$.

We let $\xi$ be a primitive element of $F$. By~\cite{Singer1938,Halberstam1963} the set $S=\{\xi K^*\}\cup \{(1+t\xi)K^*;\ t\in K\}$ of $q+1$ cosets of $K^*$ is a perfect difference set for the cyclic group $G$. We therefore define the graph $\Diff(G,S)$ to have vertex set $G$, with vertices $\xi^iK^*$ and $\xi^jK^*$ adjacent if and only if $\xi^{i+j}K^*\in S$.

Different choices of $\xi$ will in general give different perfect difference sets $S$ by this construction. However, it is proved in~\cite{Halberstam1963} that all such perfect difference sets for a given prime power $q$ are equivalent. By Lemma~\ref{lem:iso} therefore, all difference graphs obtained in this way are isomorphic, and we may denote them by $\Diff(q)$ for a given $q$.

\section{Relationship to Brown graphs}\label{sec:browngraphs}
In~\cite{Brown1966}, Brown introduced a family of graphs of diameter 2 which asymptotically approach the Moore bound for certain values of the maximum degree $\Delta$. (These graphs had previously been studied by Erd\H{o}s, R\'enyi and S\'os in a different context~\cite{Erdos1962}.) Given a prime power $q$, we define the graph $B(q)$ as follows. The vertex set of $B(q)$ is the set of points in the projective space $PG(2,q)$; equivalently, we identify a vertex with a vector $\overline{x}=(x_0,x_1,x_2)$ in $(GF(q))^3$, with not all coordinates zero, considering vectors to be the same if one is a constant multiple of the other. Two vertices in $B(q)$ represented by vectors $\overline{x}$ and $\overline{y}$ are adjacent if and only if $\overline{x}\cdot\overline{y}=0$; that is, $x_0 y_0+x_1 y_1+x_2 y_2=0$.

The properties of these graphs were studied in detail in~\cite{Bachraty2014} and we list their most relevant parameters here.
\begin{itemize}
	\item $B(q)$ has order $q^2+q+1$ and diameter 2.
	\item $B(q)$ has $q+1$ vertices of order $q$ and $q^2$ vertices of order $q+1$.
	\item For odd $q\geq 7$, the graph $B(q)$ is the largest known graph of diameter 2 and maximum degree $q+1$~\cite{miller2005moore}.
	\item For even $q$, it was shown in~\cite{Erdoes1980} that a small improvement can be made by adding a new vertex to $B(q)$ and joining it to all $q+1$ vertices of degree $q$, resulting in a $(q+1)$-regular graph of diameter 2 and order $q^2+q+2$.
\end{itemize}

The correspondence between the properties of $B(q)$ and our difference graph $\Diff(q)$ is striking. It is natural to ask whether these are in fact isomorphic. For small $q$, an explicit isomorphism can be readily determined. For example, an isomorphism between the graph $\Diff(4)$ illustrated in Figure~\ref{fig:diff4} and the Brown graph $B(4)$ is shown in Table~\ref{tab:iso}. In the table, the elements of $GF(4)$ are taken to be $0$, $1$, $\zeta$ and $\zeta^2=\zeta+1$, where $\zeta$ is a primitive element of $GF(4)^*$.

\begin{table}[h]\centering
\begin{tabular}{rlcrlcrl}
0 & $( 1, 0, 1 )$ &~~~~~& 1 & $( 1, \zeta^2, 1 )$ &~~~~~& 2 & $( 1, 1, 0 )$ \\
3 & $( 1, \zeta^2, \zeta^2 )$ &~~~~~& 4 & $( 0, 1, 0 )$ &~~~~~& 5 & $( 1, \zeta, \zeta )$ \\
6 & $( 0, 1, \zeta )$ &~~~~~& 7 & $( 1, \zeta, \zeta^2 )$ &~~~~~& 8 & $( 1, \zeta^2, \zeta )$ \\
9 & $( 1, \zeta^2, 0 )$ &~~~~~& 10 & $( 1, 0, 0 )$ &~~~~~& 11 & $( 0, 1, 1 )$ \\
12 & $( 0, 0, 1 )$ &~~~~~& 13 & $( 1, \zeta, 0 )$ &~~~~~& 14 & $( 1, 1, 1 )$ \\
15 & $( 0, 1, \zeta^2 )$ &~~~~~& 16 & $( 1, \zeta, 1 )$ &~~~~~& 17 & $( 1, 0, \zeta^2 )$ \\
18 & $( 1, 0, \zeta )$ &~~~~~& 19 & $( 1, 1, \zeta^2 )$ &~~~~~& 20 & $( 1, 1, \zeta )$ \\
\end{tabular}
\caption{An isomorphism between graphs $\mathrm{Diff}(4)$ and $B(4)$}
\label{tab:iso}
\end{table}

In the remainder of this section, we prove our main result which is that $B(q)$ and $\Diff(q)$ are isomorphic for all $q$. Throughout, we let $F=GF(q^3)$ for a prime power $q$ and let $K=GF(q)$ be the (unique) subfield of $F$ of order $q$; we let $F^*$ and $K^*$ denote the corresponding multiplicative groups. We let $\xi$ be a primitive element of $F$, and it turns out that the algebra is much simplified if the minimal polynomial of $\xi$ over $K$ has a zero quadratic term. We therefore need the following lemma.

\begin{lemma}\label{lem:minpoly}
If $q$ is a prime power other than 4, then $F=GF(q^3)$ has a primitive element with a minimal polynomial over $K=GF(q)$ of the form $x^3-(\alpha x + \beta)$ for some non-zero $\alpha,\beta\in K$.
\end{lemma}
\begin{proof}
By~\cite{Cohen2005}, there is a primitive cubic polynomial in $p(x)\in K[x]$ of the required form provided $q\neq 4$. Clearly $\beta\neq 0$ since $p$ is irreducible; and $\alpha\neq 0$ since a cube root of an element in $K$ must have multiplicative order at most $3(q-1)$ and so cannot be primitive in $F$.
\end{proof}

The idea of the proof of isomorphism is to identify the vertex sets in $\Diff(q)$ and $B(q)$ in a natural way using Singer's finite field construction of the perfect difference set from Section~\ref{sec:diffgraphs}. In $\Diff(q)$, the vertices are the elements of $G=F^*/K^*$ and in $B(q)$, the vertices are vectors of the form $\overline{x}=(x_0,x_1,x_2)$ with scalar multiples considered the same vector. By choosing a basis for $F$ as a 3-dimensional vector space over $K$, we immediately have a bijection between the two vertex sets. If we can find a $K$-basis for $F$ such that this bijection becomes a graph isomorphism, then we are done.

Before the proof of the main result we need one small lemma, which is a standard result.

\begin{lemma}\label{lem:twosquares}
Let $q$ be a prime power and let $b$ be any element of $GF(q)$. Then there exist $c,d\in GF(q)$ such that $c^2+d^2=b$.
\end{lemma}

We are now ready to prove the main result of this section.

\begin{theorem}\label{thm:iso}
Let $q$ be any prime power. Then the graphs $\Diff(q)$ and $B(q)$ are isomorphic.
\end{theorem}

\begin{proof}
If $q=4$, then an explicit isomorphism is given in Table~\ref{tab:iso}. So suppose $q\neq 4$. By Lemma~\ref{lem:minpoly} there is a primitive element $\xi$ of $F$ such that $\xi^3=\alpha\xi+\beta$ for non-zero $\alpha,\beta\in K$.

We will use the Singer difference set $S$ on $G=F^*/K^*$ given by the set of $q+1$ cosets of the form $S=\{\xi K^*\}\cup \{(1+t\xi)K^*;\ t\in K\}$. To simplify the notation, for any pair of elements $r,s\in F^*$ we will write $r\sim s$ if and only if $rK^*=sK^*$.

In the difference graph $\Diff(G,S)$, two distinct vertices $\xi^iK^*$ and $\xi^jK^*$ are adjacent if and only if $\xi^iK^* \cdot \xi^jK^*\in S$, which translates into $\xi^{i+j}\sim \xi$ or $\xi^{i+j}\sim 1+t\xi$ for some $t\in K$. Writing down $\xi^i$ and $\xi^j$ in terms of the basis $\{1,\xi,\xi^2\}$ of $F$ over $K$, one has
$\xi^i=x_0 + x_1\xi + x_2\xi^2$ and $\xi^j=y_0+y_1\xi + y_2\xi^2$ for some $x_i,y_i\in K$, $i\in\{0,1,2\}$. Now, using $\xi^3=\alpha\xi+\beta$ and $\xi^4=\alpha\xi^2+\beta\xi$, the product $\xi^i\xi^j$ evaluates to
\[ \xi^{i+j} = \gamma + \delta\xi + (x_0y_2 + x_1y_1 + x_2y_0 + \alpha x_2y_2)\xi^2 \]
where $\gamma =x_0y_0 + (x_1y_2+x_2y_1)\beta$ and $\delta =x_0y_1+x_1y_0+(x_1y_2+x_2y_1)\alpha +x_2y_2\beta$. It is now clear that the adjacency condition $\xi^{i+j}\sim \xi$ or $\xi^{i+j}\sim 1+t\xi$ for some $t\in K$ is satisfied if and only if the value of the symmetric bilinear form \[ {\mathcal B}(\overline{x},\overline{y}) = x_0y_2 + x_1y_1 + x_2y_0 + \alpha x_2y_2 \] is equal to zero for the vectors $\overline{x}=(x_0,x_1,x_2)$ and  $\overline{y}=(y_0,y_1,y_2)$ representing the elements $\xi^i$ and $\xi^j$; note that $K^*$-multiples of $\overline{x}$ and $\overline{y}$ represent the elements $\xi^iK^*$ and $\xi^jK^*$ of $F^*/K^*$. This gives an isomorphism of our difference graph $\Diff(G,S)$ onto a Brown-like graph $P(K^3,{\mathcal B})$ whose vertices are projective non-zero triples in $K^3$, with two vertices $\overline{x}K^*$ and $\overline{y}K^*$ adjacent if and only if ${\mathcal B}(\overline{x},\overline{y}) =0$.

To complete the proof, we must show that the above bilinear form $\mathcal{B}$ is projectively equivalent to the standard dot product $\mathcal{A}(\overline{x},\overline{y})=x_0 y_0+x_1 y_1+x_2 y_2$; that is to say, there is a basis change matrix $A$ which takes one to the other, up to a scalar multiple. We let $\mathcal{B}$ be represented by the symmetric matrix
\[ B = \begin{pmatrix} 0 & 0 & 1 \\ 0 & 1 & 0 \\ 1 & 0 & \alpha \end{pmatrix}  \]
so that $\mathcal{B}(\overline{x},\overline{y})=\overline{x}B\overline{y}^T$. Similarly, $\mathcal{A}$ is represented by the $3\times 3$ identity matrix $I$. So we seek a matrix $A$ such that $A^T BA=\gamma I$, for some non-zero $\gamma$.

By~\cite[Theorem 5.8]{Hirschfeld1998} for odd $q$ the bilinear form $\mathcal{B}$ is indeed projectively equivalent to ${\mathcal A}$. In~\cite{Hirschfeld1998} a method is given to explicitly construct a basis change matrix $A$. Recalling that by Lemma~\ref{lem:twosquares} there exist $c,d\in K$ with $c^2+d^2=-1$, for odd $q$ it can be checked that the following matrix $A$ satisfies $A^T BA=-I$:
\[ A = \begin{pmatrix} d-c\alpha/2  & -(c+d\alpha/2) & -(1+\alpha/2) \\ c-d & c+d & 1 \\ c & d & 1 \end{pmatrix}  \]
If $q$ is a power of $2$, then the non-zero element $\alpha\in K$ has a unique square root $\sqrt{\alpha}\in K$, and then one can take
\[ A = \begin{pmatrix} \sqrt{\alpha}  & 0 & 0 \\ 0 & 1 & 0 \\ \sqrt{\alpha^{-1}} & 0 & \sqrt{\alpha^{-1}} \end{pmatrix}  \]
In either case, the basis $(1,\xi,\xi^2)A$ is a $K$-basis for $F$ demonstrating the isomorphism between $\Diff(q)$ and $B(q)$.
\end{proof}

\section{A variation on the construction}\label{sec:variation}
As it stands, Brown's construction (and hence also our difference graph construction) may be used only to construct graphs of diameter 2 and maximum degree $\Delta=q+1$ for some prime power $q$. In~\cite{Siran2011large}, the authors address this issue by modifying the Brown graphs to have larger maximum degree, by adding edges to the basic Brown graph. In this way they are able to construct asymptotically good graphs of diameter 2 for any given value of maximum degree $\Delta$, having order equal to the order of the Brown graph corresponding to the largest prime power $q$ such that $q+1\leq\Delta$. In this section we expand the ideas of our difference graph construction in Section~\ref{sec:diffgraphs}, and show that for certain values of the maximum degree we can use an alternative to the construction in~\cite{Siran2011large}.

We begin with some necessary definitions. Let $G$ be a group and let $N$ be some proper normal subgroup of $G$. Suppose $N$ has order $n$ and $G$ has order $mn$. An $(m,n,k,\lambda)$ \emph{relative difference set} $R$ is a set of $k$ elements of $G$ with the property that every element of $G\setminus N$ occurs exactly $\lambda$ times as a difference of distinct elements $r_1,r_2\in R$, and no non-identity element of $N$ occurs at all. Our construction will use the relative difference set in the following lemma, which is easily proved.

\begin{lemma}\label{lem:rds}
Let $p$ be an odd prime, let $F$ be the field $GF(p)$ and denote the additive and multiplicative groups of $F$ by $F^+$ and $F^*$ respectively. Let $G=F^+\times F^+$ and let $N$ be the subgroup of $G$ defined by $N=\{(0,a):a\in F^+\}$. Then $R=\{(a,a^2):a\in F^+\}$ is a $(p,p,p,1)$ relative difference set for $G$ relative to $N$.
\end{lemma}

The idea of our modified construction is to use the relative difference set $R$ from Lemma~\ref{lem:rds} to define most of the adjacencies in a graph of order $q^2$, then add further edges based on the ideas of Section~\ref{sec:diffgraphs} so that the resulting graph has diameter 2.

Let $p$ be an odd prime, let $F=GF(p)$, $G=F^+\times F^+$ and let $R$ be the relative difference set in Lemma~\ref{lem:rds}. Let $\Gamma_0$ be the graph with vertex set $G$ and edges defined as follows.
\[(a,b)\sim(c,d)\iff (ab)+(c,d)\in R\]
(As usual we suppress any loops in the above definition.)

From the definition of $R$, it is immediate that two arbitrary vertices $(a,b)$ and $(c,d)$ in $\Gamma_0$ are at distance at most 2 provided $a\neq c$. We now use the construction of Section~\ref{sec:diffgraphs} to handle adjacencies between vertices where $a=c$. To do this, we require that our prime $p$ must be of the form $p=q^2+q+1$ for some prime power $q$. This motivates the final definition of our graph as follows.

Let $q$ be a prime power such that $p=q^2+q+1$ is a prime. Let $R$ be the relative difference set in Lemma~\ref{lem:rds}; let $S$ be a perfect difference set for $\Z_p$ and let $F=GF(p)$, $G=F^+\times F^+$. Let $\Gamma$ be the graph with vertex set $G$ and edges defined as follows.
\[(a,b)\sim(c,d)\iff (a,b)\neq(c,d)\text{ and }
\begin{cases}
(ab)+(c,d)\in R\\
\quad\text{or}\\
a=c\text{ and }b+d\in S
\end{cases}\]
It is easy to see that $\Gamma$ has maximum degree $\Delta=p+q+1$, order $p^2$ and diameter 2.

The graphs produced by this construction are asymptotically optimal, in the sense that the order approaches 
$\Delta^2$ as $\Delta\to\infty$. For those values of $\Delta$ for which our construction applies, we should 
note that the method  in~\cite{Siran2011large} of simply adding edges to a Brown graph will in general yield a slightly larger number of vertices; however, the construction here is new as far as we are aware.
\bigskip

\noindent{\bf Acknowledgment}~ The third author acknowledges support from the APVV Research Grants 15-0220 and 17-0428, and the VEGA Research Grants 1/0142/17 and 1/0238/19.


\end{document}